\documentclass[12pt,reqno]{amsart}

\PassOptionsToPackage{hyphens}{url}\usepackage[hidelinks]{hyperref}

\usepackage[utf8]{inputenc}
\usepackage[british]{babel}
\usepackage{amsmath}
\usepackage{amsthm}
\usepackage{amssymb}
\usepackage{mathrsfs}
\usepackage{indentfirst}
\usepackage{stmaryrd}
\usepackage{bbm}
\usepackage{xspace, xcolor}
\usepackage{thmtools}
\usepackage{thm-restate}
\usepackage{calc}
\usepackage{fullpage}

\usepackage{bookmark}
\usepackage{enumitem}
\usepackage{xcolor}
\usepackage[normalem]{ulem}
\usepackage{comment}
\usepackage[noabbrev, nameinlink]{cleveref}

\usepackage[numbers,longnamesfirst]{natbib}
\usepackage{graphicx}

\usepackage{graphicx}
\graphicspath{{./Pictures/} }
\usepackage{subcaption}

\usepackage{fullpage}
\usepackage{setspace}

\newtheorem{theorem}{Theorem}[section]

\newtheorem{claim}[theorem]{Claim}

\theoremstyle{definition}

\newcommand{\PP}{\mathbb{P}}
\newcommand{\EE}{\mathbb{E}}

\newcommand*{\ie}{\text{i.e.}}

\newcommand*{\N}{\mathbb{N}}

\newcommand*{\F}{\mathbb{F}}

\newcommand{\ex}{\textup{ex}}

\newcommand{\ceil}[1]{\left\lceil #1 \right\rceil}

\renewcommand{\leq}{\leqslant}
\renewcommand{\le}{\leqslant}
\renewcommand{\geq}{\geqslant}
\renewcommand{\ge}{\geqslant}

\DeclareMathOperator{\Span}{Span}

\title{$C_{10}$ \lowercase{has positive }T\lowercase{ur\'an density in the hypercube}}

\author{A\lowercase{lexandr }G\lowercase{rebennikov, }J\lowercase{o\~ao }P\lowercase{edro }M\lowercase{arciano}}

\address{IMPA, Estrada Dona Castorina 110, Jardim Bot\^anico, Rio de Janeiro, RJ, Brazil}
\email{joao.marciano@impa.br}

\address{IMPA, Estrada Dona Castorina 110, Jardim Bot\^anico, Rio de Janeiro, RJ, Brazil}
\email{sagresash@yandex.ru}


\begin{document}

\begin{abstract}
    The $n$-dimensional hypercube $Q_n$ is a graph with vertex set $\{0,1\}^n$ such that there is an edge between two vertices if and only if they differ in exactly one coordinate.
    For any graph $H$, define $\textup{ex}(Q_n,H)$ to be the maximum number of edges of a subgraph of $Q_n$ without a copy of $H$.
    In this short note, we prove that for any $n \in \mathbb{N}$
	\[
		\textup{ex}(Q_n, C_{10}) > 0.036 \cdot e(Q_n),
	\] 
	where $e(Q_n)$ is the number of edges of $Q_n$.
    Our construction is strongly inspired by the recent breakthrough work of Ellis, Ivan, and Leader, who showed that ``daisy'' hypergraphs have positive Tur\'{a}n density with an extremely clever and simple linear-algebraic argument.
\end{abstract}

\maketitle

\section{Introduction}
For each $n \in \N$, define the $n$-dimensional hypercube $Q_n$ to be the graph with vertex set $\{0, 1\}^n$ such that there is an edge between two vertices if and only if they differ in exactly one coordinate. 
\citet{erdos1984some} initiated the study of $\ex(Q_n, H)$, the maximum number of edges in an $H$-free subgraph of $Q_n$, in the special case when $H$ is an even cycle.
We say that $H$ has positive Turán density in the hypercube if there is some constant $\alpha>0$ such that for every $n \in \N$ 
\[\ex(Q_n,H) \geq \alpha \cdot e(Q_n),\]
where $e(Q_n)$ denotes the number of edges of $Q_n$.

We identify the vertices of $Q_n$ with the subsets of $[n]:=\{1,\ldots,n\}$ in the usual way.
For any $r \le n$, we define $L_r(n)$, the $r$-th (edge) layer of $Q_n$, to be the subgraph of $Q_n$ formed by the edges between $\binom{[n]}{r-1}$ and $\binom{[n]}{r}$, where
\[\binom{[n]}{r} := \big\{S \subset [n]  :  |S|=r\big\}.\]
It is easy to see that $C_4 \not \subset L_r(n)$, thus taking every second layer gives that
\[\ex(Q_n,C_4) \geq \frac{1}{2} \cdot e(Q_n).\]
\citet{erdos1990some} offered $\$100$ for an answer to whether this bound is optimal, and it remains open until the present date. 
The best known upper bound, due to \citet*{baber2012turan}, is approximately $0.60318 \cdot e(Q_n)$ (see \citep{chung1992subgraphs, thomason2009bounding, balogh2014upper} for previous upper bounds).

\citet{erdos1984some} suggested that longer even cycles might not have positive Turán density, but this was shown to be false for $C_6$ by \citet{chung1992subgraphs} and \citet{brouwer1993highly}.
On the other hand, it was proved by \citet{chung1992subgraphs} for every even $t \geq 4$, and by \citeauthor{furedi200914} \citep{furedi200914, furedi2011even} for every odd $t \geq 7$, that
\[\ex(Q_n,C_{2t})=o(e(Q_n)),\]
and a unified proof for all of the above cases was given by \citet{conlon2010extremal} (the best known upper bounds for $t$ large are due to \citet{chung1992subgraphs} when $t$ is even, and \citeauthor{Tomon2024Robust} \citep[Theorem 4.1]{Tomon2024Robust} when $t$ is odd).
Thus, the only case for which the problem remained open was $C_{10}$. 
The main result of this paper completes the picture by showing that $C_{10}$ has positive Turán density in the hypercube.

\begin{theorem}\label{thm:main}
    For all $n \in \N$
    \begin{equation*}
        \ex(Q_n, C_{10}) > \frac{c}{8} \cdot e(Q_n),
    \end{equation*} 
    where 
    \[c=\prod_{k=1}^\infty \bigg(1-\frac{1}{2^k}\bigg)>0.288.\]
\end{theorem}

We remark that the graph we construct to prove \Cref{thm:main} is free of $C_4$, $C_6$ and $C_{10}$ simultaneously.
The previous best known lower bound in this case was
\[ \ex(Q_n,C_{10}) = \Omega\left(\frac{e(Q_n)}{(\log n)^\alpha}\right) \]
for some constant $\alpha>0$, recently shown by \citet[Theorem 4]{axenovich2024graphs}.

The Ramsey problem for even cycles in the hypercube has also attracted a great deal of attention over the years. 
In particular, \citet{chung1992subgraphs} and \citet*{brouwer1993highly} found $4$-colourings of the hypercube without a monochromatic $C_6$, and \citet{conder1993hexagon} found a $3$-colouring with the same property, implying the best known lower bound for this case of $\ex(Q_n, C_6) \geq \frac{1}{3} e(Q_n)$.
However, it was shown by \citet{alon2006ramsey} that for every fixed $k$ and sufficiently large $n$, any $k$-colouring of the edges of $Q_n$ contains a monochromatic copy of $C_{10}$. 
Our result therefore gives, to the best of our knowledge, the first example of a graph with the hypercube Ramsey property and positive Turán density.
The best upper bound on its Turán density we are aware of is $\ex(Q_n,C_{10}) \leq (\frac{1}{\sqrt{2}} + o(1)) \cdot e(Q_n)$ by \citet[Theorem 3.3]{axenovich2006note}.

Our proof of \Cref{thm:main} is heavily inspired by a recent breakthrough work of \citet*{ellis2024tur}, who gave an extremely clever and simple linear-algebraic construction which shows that ``daisies'' have positive Turán density in hypergraphs, disproving a conjecture of \citet{bollobas2011daisies} and \citet{bukh2008set}. The argument we use is essentially a slight modification of their approach.

We will need another important idea, first observed by \citet*{alon2006ramsey} and \citet*{axenovich2006note}: there exists a $4$-colouring of $E(Q_n)$ without monochromatic \textit{induced} copies of $C_{10}$. Furthermore, one can find such a colouring that uses only two different colours on any given layer of the hypercube (see the explicit construction in \cite{axenovich2006note}). Restricting this colouring to a single layer, we have the following statement.

\begin{theorem}[\cite{axenovich2006note}] \label{thm:colouring}
	For any $r, n \in \N$ with $r \leq n$, there exists a $2$-colouring of the edges of $L_r(n)$ that does not contain a monochromatic induced copy of $C_{10}$.
\end{theorem}

Now, the core of the proof is the following theorem.

\begin{theorem}\label{thm:C6-free}
    For any $r, n \in \N$ with $r \leq n$, there exists a $C_6$-free induced subgraph $G_r$ of $L_r(n)$ with $$e(G_r) > \frac{c}{2} \cdot e(L_r(n)).$$
\end{theorem}

\begin{proof}[\textbf{Proof of \Cref{thm:main} using \Cref{thm:C6-free}}]
	For each odd $r \leq n$, let $G_r$ be the $C_6$-free induced subgraph of $L_r(n)$ given by \Cref{thm:C6-free}. Since $G_r$ is a subgraph of $L_r(n)$, it is also bipartite and $C_4$-free. Moreover, every copy of $C_{10}$ in $G_r$ is an induced subgraph of $L_r(n)$: indeed, if it has a chord in $L_r(n)$, then this chord in fact lies in $G_r$ (because $G_r$ is an induced subgraph of $L_r(n)$) and hence creates a copy of $C_4$ or $C_6$ in $G_r$, leading to a contradiction.
	
	Let $G'_r$ be the largest colour class of $G_r$ in the $2$-colouring of $L_r(n)$ given by \Cref{thm:colouring}, so that $e(G'_r) \ge e(G_r)/2$. Since this colouring has no monochromatic induced copies of $C_{10}$, $G'_r$ is $C_{10}$-free.
    Take $G \subset Q_n$ to be the union of the graphs $G'_r$ for every odd $r \leq n$, which clearly have disjoint vertex sets. Then $G$ is also $C_{10}$-free, and
    \[e(G) \geq \frac{1}{2} \sum_{s=1}^{\ceil{n/2}} e(G_{2s-1}) > \frac{c}{4} \sum_{s=1}^{\ceil{n/2}} e\big(L_{2s-1}(n)\big) \geq \frac{c}{8} \cdot e(Q_n), \]
    which concludes the proof.
\end{proof}

\textbf{Corrections made after publication.} In the journal version of this paper, \Cref{thm:main} gives a slightly worse bound $\ex(Q_n, C_{10}) > \frac{c}{12} \cdot e(Q_n)$. Its proof is identical to the one in the current version, except that it relies on \cite[Lemma 16]{axenovich2024graphs} instead of \Cref{thm:colouring}. In the journal version of \cite{axenovich2024graphs}, this Lemma 16 states that $\ex(Q_n,C_{10}) \geq \frac{1}{3} \cdot \ex^*(Q_n,C_6^{-})$,\footnote{Here $C_6^-$ denotes any subgraph of $Q_n$ obtained by removing an edge from a copy of $C_6$ in $Q_n$, and $\ex^*(Q_n,C_6^{-})$ is the maximum number of edges in a subgraph of $Q_n$ containing no $C_6^-$. Note that every $C_6^-$ is a path with $5$ edges, but not every such path is a $C_6^-$.} but the proof written there is incomplete. 
The later version\footnote{See \href{https://arxiv.org/abs/2303.15529v4}{arXiv:2303.15529v4}.} of \cite{axenovich2024graphs} gives a corrected argument (also based on the colouring from \cite{axenovich2006note}) for a slightly weaker bound $\ex(Q_n,C_{10}) \geq \frac{1}{4} \cdot \ex^*(Q_n,C_6^{-})$. So, plugging the corrected version of this lemma into the old proof of \Cref{thm:main} would yield an even worse bound $\ex(Q_n, C_{10}) > \frac{c}{16} \cdot e(Q_n)$. However, by observing that the colouring from \cite{axenovich2006note} uses only two different colours in each layer, we win an extra factor of $2$ and obtain the bound stated in \Cref{thm:main}.

\section{Proof of Theorem~\ref{thm:C6-free}}

Similarly to the argument used in \citep{ellis2024tur}, we consider the vector space $\F_2^r$. 
Also we fix a nonzero vector $v_0 \in \F_2^r$. 
For each $i \in [n]$, pick a vector $v_i \in \mathbb{F}_2^r \setminus \{0\}$ independently and uniformly at random. 
For a subset $S$ of $[n]$ define the multiset of vectors $M(S) = \{v_i : i \in S\}$.

Recall that each vertex of our hypercube layer $L_r(n)$ is identified with a certain subset of $[n]$ of size $r$ or $r-1$.
Define the sets
\[
	B_r = \bigg\{S \in {[n] \choose r}  :  M(S) \text{ forms a basis of } \F_2^r \bigg\},
\]
\[
	B_{r-1} = \bigg\{S \in {[n] \choose {r-1}} :  \{v_0\} \cup M(S) \text{ forms a basis of } \F_2^r \bigg\},
\]
and denote by $G_r$ the induced subgraph of the layer $L_r(n)$ on $B_{r-1} \cup B_r$.

We will show that this (random) graph $G_r$ is $C_6$-free (deterministically) and has a large expected number of edges, and therefore has the desired properties with positive probability.

\begin{claim}\label{claim:1}
	$G_r$ is $C_6$-free.
\end{claim}

\begin{proof}	
    Suppose there is a copy of $C_6$ in $G_r$. 
    There are exactly three different coordinates flipped by its edges, so it must form the middle layer of some $3$-dimensional subcube of $Q_n$. 
    More precisely, by permuting the coordinates, we may assume that this copy of $C_6$ is exactly of the form $L_r(n)\big[A_1\cup A_2\big]$, where
    \[A_1=\big\{ \{i\} \cup I : 1 \leq i \leq 3\big\} \qquad \text{and} \qquad A_2=\big\{\{i,j\} \cup I : 1 \leq i<j \leq 3 \big\}\]
    for some $I \subset [n] \setminus \{1,2,3\}$ with $|I|=r-2$.
    
	Observe that, since $\{1,2\} \cup I \in A_2 \subset B_r$, the collection of vectors $\{v_1, v_2\} \cup M(I)$ forms a basis of $\F_2^r$, so, in particular, the vectors in $M(I)$ are linearly independent.
	Therefore, after taking the quotient of $\F_2^r$ by $\Span(M(I))$ we obtain a vector space $V$ isomorphic to $\F_2^2$.
	
	Define $x_0, x_1, x_2, x_3 \in V$ to be the images of $v_0, v_1, v_2, v_3$, respectively, under the quotient map. 
    Since $A_2 \subset B_{r}$ we have that
    \[
    \{v_i, v_j\}\cup M(I) \text{ forms a basis of $\F_2^r$ for any $1 \le i < j \le 3$,}
    \]
    \ie,
	\[
		\{x_i, x_j\} \text{ forms a basis of $V$ for any $1 \le i < j \le 3$.}
	\]
	Similarly, from $A_1 \subset B_{r-1}$, by the same argument as above, we obtain
	\[
		\{x_0, x_i\} \text{ forms a basis of $V$ for any $1 \le i \le 3$.}
	\]
	In particular, this implies that $x_i \neq 0$ for each $i \in \{0,1,2,3\}$ and that $x_i \ne x_j$ for any $i \neq j$. 
    But $|V \setminus \{0\}| = 3$, which yields a contradiction.
\end{proof}

\begin{claim}\label{claim:2}
	$\EE\big[e(G_r)\big] > \dfrac{c}{2} \cdot e\big(L_r(n)\big)$.
\end{claim}
\begin{proof}
	Consider an edge of $L_r(n)$ connecting two sets $x = \{j_1, \ldots, j_{r-1}\}$ and $y = \{j_1, \ldots, j_r\}$. Define the vector spaces 
	\[
		V_0 = \Span\{v_0\}, \; V_k = \Span\{v_0, v_{j_1}, \ldots, v_{j_k}\} \; \text{ for $1 \le k \le r-1$,}
	\]
	\[
		\text{ and }V_r = \Span\{v_{j_1}, \ldots, v_{j_r}\}.
	\]
	Note that 
	\begin{equation}
	\label{2.2-basic}
		\PP\big(xy \in G_r\big) = \PP\big(x \in B_{r-1} \text{ and } y \in B_r\big) = \PP\big(\dim V_{r-1} = \dim V_r = r\big).
	\end{equation}
	Observe that, if the vectors $v_0, v_{j_1}, \ldots, v_{j_{k-1}}$ (for $1 \le k \le r-1$) are fixed and linearly independent, then there are exactly $2^r - 2^k$ choices for the next vector $v_{j_k}$ such that $v_0, v_{j_1}, \ldots, v_{j_k}$ are still linearly independent, as it is equivalent to the condition $v_{j_k} \notin V_{k-1}$. 
	Similarly, if the vectors $v_{j_1}, \ldots, v_{j_{r-1}}$ are fixed and linearly independent, then there are $2^r - 2^{r-1}$ choices for the last vector $v_{j_r}$ such that $v_{j_1}, \ldots, v_{j_r}$ are still linearly independent, as it is equivalent to the condition $v_{j_r} \notin \Span\{v_{j_1}, \ldots, v_{j_{r-1}}\}$. 
    Therefore, by \eqref{2.2-basic}, we have
	\begin{align*}
		\PP\big(xy \in G_r\big) &= \Bigg(\prod_{k = 1}^{r-1} \PP\big(v_{j_k} \notin V_{k-1} \;|\;  \dim V_{k-1} = k\big) \Bigg) \cdot \PP\big(v_{j_r} \notin \Span\{v_{j_1}, \ldots, v_{j_{r-1}}\} \; | \; \dim V_{r-1} = r\big) \\
		&= \Bigg(\prod_{k = 1}^{r-1} \frac{2^r - 2^k}{2^r - 1} \Bigg) \cdot \frac{2^r - 2^{r-1}}{2^r - 1} > \frac{1}{2} \prod_{k = 1}^\infty \bigg(1 - \frac{1}{2^k} \bigg) = \frac{c}{2}.
	\end{align*}
	The claim now follows from linearity of expectation.
\end{proof}
As observed above, it follows from Claims \ref{claim:1} and \ref{claim:2} that there exists a choice of the vectors $v_1, \ldots, v_n$ such that 
\[e(G_r) > \frac{c}{2} \cdot e\big( L_r(n)\big)\]
and $G_r$ is $C_6$-free, as required. \qed

\section*{Acknowledgements}

We are grateful to Maria Axenovich for the discussion of the corrections (listed at the end of the introduction) and a suggestion to use Theorem 1.2 in its current form.
We would also like to thank the anonymous referees for their valuable suggestions. This study was financed in part by the Coordenação de Aperfeiçoamento de Pessoal de Nível Superior, Brasil (CAPES).

\bibliographystyle{plainnat}
\bibliography{main}

\end{document}